\documentclass{amsart}
\usepackage{latexsym}
\usepackage{amsmath, amssymb, amsfonts, amsxtra, amsthm}

\newcommand {\rea}{\mathbb{R}}
\newcommand {\BBZ}{\mathbb{Z}}
\def\<{\left\langle}
\def\>{\right\rangle}
\newcommand {\calI}{\mathcal{I}}

\newtheorem{theorem}{Theorem}[section]

\newtheorem{proposition}[theorem]{Proposition}

\newtheorem{lemma}[theorem]{Lemma}

\author{Richard Oberlin}
\address{
Mathematics Department\\
Louisiana State University\\
Baton Rouge, LA}
\email{oberlin@math.lsu.edu}
\thanks{The author is supported in part by NSF Grant DMS-1068523.}

\subjclass{Primary 42A45}
\begin{document}

\title[Compositions of maximal operators with singular integrals]{Estimates for compositions of maximal operators with singular integrals}

\begin{abstract}
We prove weak-type $(1,1)$ estimates for compositions of maximal operators with singular integrals. Our main object of interest is the operator $\Delta^*\Psi$ where $\Delta^*$ is Bourgain's maximal multiplier operator and $\Psi$ is the sum of several modulated singular integrals; here our method yields a significantly improved bound for the $L^q$ operator norm when $1 < q < 2.$ We also consider associated variation-norm estimates.
\end{abstract}

\maketitle

\section{Introduction}

Let $\Omega$ be the set of all dyadic subintervals of $\rea$ and let $\{\phi_{\omega}\}_{\omega \in \Omega}$ be a collection of smooth functions, each adapted to $\omega$ in the sense that $\phi_{\omega}$ is supported on $\omega$ and the quantity 
\begin{equation} \label{necessarysmoothness}
D_M = \sup_{\omega \in \Omega} |\omega|^{M} \|\phi_{\omega}^{(M)}\|_{L^{\infty}} 
\end{equation}
is finite where $\phi_{\omega}^{(M)}$ denotes the $M$'th derivative of $\phi_{\omega}$ and $M$ is some large number which depends on the quantity $\epsilon$ below. Let $\Xi$ be a finite collection of real numbers. For each integer $k$ consider the operator 
\[
\Delta_k[f] = \sum_{\substack{\omega \in \Omega : |\omega| = 2^{-k} \\ \omega \cap \Xi \neq \emptyset}} \check{\phi}_{\omega} * f.
\]
One then forms the maximal operator
\[
\Delta^*[f](x) = \sup_k |\Delta_k[f](x)|.
\]
Bounds for operators similar to $\Delta^*$ were originally studied by Bourgain \cite{bourgain89pet}, and have since proven to be useful for approaching many problems in time-frequency analysis and pointwise convergence for ergodic systems. 

It follows from the method of \cite{nazarov10czd}, see also \cite{demeter09osm}, that for $1 < q \leq 2$ and $r > 2$ 
\begin{equation} \label{LqBourgain}
\|\Delta^*[f]\|_{L^q} \leq C_{q,r} (1 + \log |\Xi|)|\Xi|^{\frac{1}{q} - \frac{1}{r}} (D_M + \sup_{\xi \in \Xi} \|\sum_{\substack{\omega \in \Omega : |\omega| = 2^k }} \phi_\omega(\xi)\|_{V^r_k}) \|f\|_{L^q}
\end{equation}
where $\|\cdot\|_{V^r}$ is the $r$-variation norm (see below). The bound above is proven by establishing a weak-type estimate at $L^1$ and interpolating it with the $L^2$ bound which was originally proven in \cite{demeter08bdr}. 

Our focus here will be on studying $L^q$ bounds for operators formed by composing $\Delta^*$ with certain Fourier-multipliers. Let $\Upsilon$ be a finite set of disjoint (not necessarily dyadic) subintervals of $\rea$ and let $\{\psi_{\upsilon}\}_{\upsilon \in \Upsilon}$ be a collection of functions such that each $\psi_{\upsilon}$ is supported on $\upsilon$. We then write 
\[
\Psi[f] = \sum_{\upsilon \in \Upsilon} (\psi_{\upsilon} \hat{f}) \check{\ }.
\]
It was proven by Coifman, Rubio de Francia, and Semmes \cite{coifman88mdf}, see also \cite{tao01emt}, that for $r \geq 2$, $\frac{1}{q} - \frac{1}{2} < \frac{1}{r}$, and $\epsilon > 0$
\begin{equation} \label{CRSbound}
\|\Psi[f]\|_{L^q} \leq C_{q,r,\epsilon} |\Upsilon|^{\frac{1}{q} - \frac{1}{2} + \epsilon} \sup_{\upsilon \in \Upsilon}\|\psi_{\upsilon}\|_{V^r} \|f\|_{L^q}.
\end{equation}

Separate applications of \eqref{LqBourgain} and \eqref{CRSbound} give a bound for the operator norm of $\Delta^*\Psi$ which is on the order of $|\Xi|^{\frac{1}{q} - \frac{1}{r} + \epsilon}|\Upsilon|^{\frac{1}{q} - \frac{1}{2} + \epsilon} .$  The goal of this paper is to improve the norm estimate to $(|\Xi| + |\Upsilon|)^{\frac{1}{q} - \frac{1}{r} + \epsilon}$ (to put this in context, we are mostly interested in the case when $r$ is close to 2 and $|\Xi|$ and $|\Upsilon|$ are comparable). Specifically, we will demonstrate

\begin{theorem} \label{maintheorem}
Suppose $1 < q < 2$, $2 < r < 2q$, and $\epsilon > 0$. Then
\begin{multline}
\|\Delta^*[\Psi[f]]\|_{L^q} \leq C_{q,r,\epsilon}  
(|\Xi| + |\Upsilon|)^{\frac{1}{q} - \frac{1}{r} + \epsilon} \\
 (D_M + \sup_{\xi \in \Xi} \|\sum_{\substack{\omega \in \Omega : |\omega| = 2^k }} \phi_\omega(\xi)\|_{V^r_k}) \sup_{\upsilon \in \Upsilon} \|\psi_{\upsilon}\|_{V^r} \|f\|_{L^q}.
\end{multline}
\end{theorem}

Using the method of \cite{coifman88mdf}, where the functions $\psi_{\upsilon}$ are efficiently decomposed into sums of step functions, Theorem \ref{maintheorem} will follow from the case when each $\psi_{\upsilon}$ is a constant multiple of $1_{\upsilon}.$
Specializing further to the situation $|\Xi| = |\Upsilon| = 1$, the resulting operator bears some resemblance to the composition of a maximal averaging operator with the Hilbert transform
\[
H[f](x) = p.v. \int f(x - y)\frac{1}{y}.
\] 
Through separate applications of the standard bounds for the Hilbert-transform and Hardy-Littlewood maximal operator, one sees that this composition is bounded for values of $q$ strictly between $1$ and $\infty$. Our method, however, will require a weak-type estimate at $q=1$ and we provide a simple proof of such an estimate, as a model for the general case, in Section \ref{sectionsfc}. 
  
Our main motivation for considering Theorem \ref{maintheorem} is its connection to the return times conjecture for the truncated Hilbert transform. Specifically, our aim is to extend a pointwise convergence result from \cite{demeter08bdr} for functions $g \in L^2$ to exponents $q$ below $2$. In \cite{oberlin11awm}, it is shown that such an extension was possible for the Walsh model of the problem, and a norm improvement, as in Theorem \ref{maintheorem}, for a Walsh-analogue of $\Delta^*\Psi$ was a key ingredient in the proof. We thus view the current work as progress towards obtaining the desired pointwise convergence result, however it is not completely clear at present whether Theorem \ref{maintheorem} is strong enough. Ideally, as in \cite{demeter08bdr}, one would like to take each function $\phi_{\omega}$ to be a constant multiple of $1_\omega$; without significant refinements, our proof does not permit this for $q$ below 2, even for weaker bounds such as \eqref{LqBourgain}. However, in the case of the return times theorem for averages \cite{demeter09irt}, it was shown that one can make due with smooth $\phi_\omega$ and we hope that the same might hold true for the truncated Hilbert transform. 

Although, to simplify the exposition, we will focus on estimates for the maximal operator $\Delta^*$, a refinement of our technique permits a variation-norm analogue of Theorem \ref{maintheorem} (see \cite{nazarov10czd} for a corresponding variation-norm version of \eqref{LqBourgain}). The details will be given in Section \ref{vnsection}, where we establish
\begin{theorem} \label{varmaintheorem}
Suppose $1 < q < 2 < r < s$ and $\epsilon > 0$ satisfy $(\frac{1}{2} - \frac{1}{r})\frac{2}{s-2} + \frac{1}{q} - \frac{1}{2} < \frac{1}{r}$ and $\epsilon > 0$. Then
\begin{multline}
\|\Delta_k[\Psi[f]](x)\|_{L^q_x(V^s_k)} \leq C_{q,r,s,\epsilon}  
(|\Xi| + |\Upsilon|)^{(\frac{1}{2} - \frac{1}{r})\frac{s}{s-2} + \frac{1}{q} - \frac{1}{2} + \epsilon} \\
 (D_M + \sup_{\xi \in \Xi} \|\sum_{\substack{\omega \in \Omega : |\omega| = 2^k }} \phi_\omega(\xi)\|_{V^r_k}) \sup_{\upsilon \in \Upsilon} \|\psi_{\upsilon}\|_{V^r} \|f\|_{L^q}.
\end{multline}
\end{theorem}

\subsection{Notation guide}
The Fourier and inverse Fourier transforms will be written $\hat{\ }$ and $\check{\ }$ respectively. We use $|\cdot|$ to denote the cardinality or Lebesgue measure of a set, or the modulus of a complex number; hopefully the meaning is clear from context. The characteristic function of a set $E$ will be written $1_E.$ Given an exponent $1 \leq r < \infty$ and a function $f$ on $\rea$ we let $\|f\|_{V^r}$ denote the $r$-variation norm of $f$
\[
\|f\|_{V^r} = \|f\|_{L^\infty} + \sup_{N, \xi_0 < \dots < \xi_N}  \left(\sum_{j = 1}^N |f(\xi_j) - f(\xi_{j-1})|^r\right)^{1/r}
\]
where the supremum is over all strictly increasing finite length sequences of real numbers. We will also apply variation-norms to functions defined on the integers by restricting the range of the sequences. When $r = \infty$ we replace the $\ell^r$ norm by the $\ell^{\infty}$ norm and essentially recover the $L^{\infty}$ norm.

\section{The single frequency case} \label{sectionsfc}

Here we give a proof of:

\begin{theorem} \label{onefrequencytheorem}
Let $\phi$ be a Schwartz function such that $\hat{\phi}$ is compactly supported, and let $M_{\phi}$ be the associated maximal averaging operator
\[
M_{\phi}[f](x) = \sup_{k \in \BBZ} |\int f(x - y) 2^{-k}\phi(2^{-k}y)\ dy|.
\]
Then 
\begin{equation} \label{oftconclusion}
\|M_{\phi}[H[f]]\|_{L^{1,\infty}} \leq C_{\phi} \|f\|_{L^1}
\end{equation}
where $\|\cdot\|_{L^{1,\infty}}$ denotes the weak $L^1$ Lorentz norm. 
\end{theorem}

 Due to the homogeneity of the multiplier defining $H$ (which is not essential for our proof), $M_\phi H$ coincides with a maximally-dilated multiplier operator and the bound above has been known at least since \cite{dappa85omf}.

We note that (in contrast with the $q > 1$ case) without utilizing the cancellation in $M_{\phi}$, any attempt at bounding $M_{\phi}[H[\cdot]]$ at $L^1$ fails utterly. Indeed, it is well known, see for example \cite{titchmarsh48itf} Section 5.16, that there are functions $f \in L^1$ such that $H[f]$ is not locally in $L^1.$ Given such an $f$ we then have $M_\phi[|H[f]|]$ identically infinite for any nonnegative $\phi$ which is nonzero on a neighborhood of the origin. 

\begin{proof}[Proof of Theorem \ref{onefrequencytheorem}]
Without loss of generality assume that $\hat{\phi}$ is supported on $[-1/2,1/2]$. Let $\hat{\psi}$ be a smooth function supported on $[-2,-1/2]\cup[1/2,2]$ such that 
\[
\sum_{j = - \infty}^{\infty} \hat{\psi}(2^j \cdot) = 1_{(0,\infty)} - 1_{(-\infty,0)}.
\]
For Schwartz functions $f$, we then have
\[
M_{\phi}[H[f]] = \sup_{k \in \BBZ} |\sum_{j = -\infty}^{\infty} \phi_k * \psi_j * f|
\]
where $\phi_k = 2^{-k}\phi(2^{-k}\cdot)$ and similarly for $\psi_j.$
By a standard approximation argument, it suffices to prove \eqref{oftconclusion} with the supremum and the sum on the right side above only ranging over finite sets of integers, provided that the constant $C_{\phi}$ is independent of these sets. 

The resulting operator is bounded on $L^2$ by the Hardy-Littlewood maximal theorem. Thus, following the Calder\'{o}n-Zygmund method it suffices to show that for each interval $I$ and mean-zero $L^1$ function $b$ supported on $I$, we have
\begin{equation} \label{CZaddup}
\|\sup_{k} |\sum_{j} \phi_k * \psi_j * b|\|_{L^1((3I)^c)} \leq C_{\phi} \|b\|_{L^1(I)}.
\end{equation}
Using the support properties of $\hat{\phi}$ and $\hat{\psi}$, we see that the left side above
\begin{align*}
&=\| \sup_k |\sum_{j > k} \phi_k * \psi_j * b| \|_{L^1((3I)^c)} \\
&\leq \sum_{j} \| \sup_{k < j} |\phi_k * \psi_j * b| \|_{L^1((3I)^c)}.
\end{align*}
For each $j$, the pointwise estimates
\[
|\phi_k * \psi_j(x)| \leq C_\phi 2^{-j}(1 + |2^{-j}x|)^{-2} 
\]
and 
\[
|\frac{d}{dx}\phi_k * \psi_j(x)| \leq C_\phi 2^{-2j}(1 + |2^{-j}x|)^{-2} 
\]
hold uniformly in $k < j$.
We thus have 
\[
\| \sup_{k < j} |\phi_k * \psi_j * b| \|_{L^1((3I)^c)} \leq C_{\phi}\|2^{-j}(1 + |2^{-j}\cdot|)^{-2} *|b|\|_{L^1((3I)^c)}
\]
and, for $2^j > |I|$
\[
\| \sup_{k < j} |\phi_k * \psi_j * b| \|_{L^1((3I)^c)} \leq |I| 2^{-j} C_{\phi}\|2^{-j}(1 + |2^{-j}\cdot|)^{-2} *|b|\|_{L^1}
\]
which then give \eqref{CZaddup} in the usual way.
\end{proof}

\section{Proof of Theorem \ref{maintheorem}}

Using the following lemma, which was proven in \cite{coifman88mdf} (see also \cite{lacey07irt}), we will show that to establish Theorem \ref{maintheorem} it suffices to consider the special case, Proposition \ref{maintheoremforchar} below, where the functions $\psi_{\upsilon}$ are constant. 

\begin{lemma} \label{intervaldecompositionlemma}
Let $\psi$ be a compactly supported function of bounded $r$-variation for some $1 \leq r < \infty.$ Then for each integer $j \geq 0$, one can find a collection $\calI_j$ of pairwise disjoint intervals and coefficients $\{c_{I}\}_{I \in \calI_j}$ so that $|\calI_j| \leq 2^j$, $|c_{I}| \leq 2^{-j/r} \|\psi\|_{V_r} $, and
\[
\psi = \sum_{j \geq 0} \sum_{I \in \calI_j} c_{I} 1_{I} 
\]
where the sum in $j$ converges uniformly.
\end{lemma}

\begin{proposition} \label{maintheoremforchar}
Suppose $1 < q < 2$, $r > 2$, and $\epsilon > 0$. For each finite collection $\Upsilon$ of disjoint intervals and collection of coefficients $\{c_{\upsilon}\}_{\upsilon \in \Upsilon}$ 
\begin{multline} \label{mtfcbound}
\|\Delta^*[\sum_{\upsilon \in \Upsilon} (c_{\upsilon} 1_{\upsilon}\hat{f})\check{\ }]\|_{L^q} \leq C_{q,r,\epsilon}  
(|\Xi| + |\Upsilon|)^{\frac{1}{q} - \frac{1}{r} + \epsilon} \\
 (D_M + \sup_{\xi \in \Xi} \|\sum_{\substack{\omega \in \Omega : |\omega| = 2^k }} \phi_\omega(\xi)\|_{V^r_k}) \sup_{\upsilon \in \Upsilon} |c_{\upsilon}| \|f\|_{L^q}.
\end{multline}
\end{proposition}

\begin{proof}[Proof of Theorem \ref{maintheorem} assuming Proposition \ref{maintheoremforchar}]
After a limiting argument, one may assume that all intervals in $\Upsilon$ have finite length. Applying Lemma \ref{intervaldecompositionlemma} to each $\psi_{\upsilon}$ we obtain for $j \geq 0$ a collection $\calI_{\upsilon,j}$ of at most $2^j$ pairwise disjoint subintervals of $\upsilon$ and coefficients $\{c_I\}_{I \in \calI_{\upsilon,j}}$ so that 
\[
\psi_{\upsilon} = \sum_{j \geq 0} \sum_{I \in \calI_{\upsilon,j}} c_{I} 1_I.
\]
Then
\[
\|\Delta^*[\Psi[f]]\|_{L^q}   
\leq \sum_{j \geq 0} \|\Delta^*[\sum_{\upsilon \in \Upsilon} \sum_{I \in \calI_{\upsilon,j}}(c_I 1_I \hat{f})\check{\ }]\|_{L^q}.
\]
Applying Proposition \ref{maintheoremforchar} to the collection of pairwise disjoint intervals $\bigcup_{\upsilon \in \Upsilon} \calI_{\upsilon,j}$ we see that each term on the right above is 
\begin{align*}
&\leq C_{q,r,\epsilon}  (|\Xi| + 2^j|\Upsilon|)^{\frac{1}{q} - \frac{1}{r} + \epsilon} 
 \sup_{\xi \in \Xi} \|\sum_{\substack{\omega \in \Omega : |\omega| = 2^k }} \phi_\omega(\xi)\|_{V^r_k} \sup_{\upsilon \in \Upsilon, I \in \calI_{\upsilon,j}} |c_{I}| \|f\|_{L^q} \\
&\leq C_{q,r,\epsilon}  2^{j(\frac{1}{q} - \frac{1}{r} + \epsilon)}(|\Xi| +|\Upsilon|)^{\frac{1}{q} - \frac{1}{r} + \epsilon} 
 \sup_{\xi \in \Xi} \|\sum_{\substack{\omega \in \Omega : |\omega| = 2^k }} \phi_\omega(\xi)\|_{V^r_k} 2^{-\frac{j}{r}} \sup_{\upsilon \in \Upsilon} \|\psi_{\upsilon}\|_{V^r} \|f\|_{L^q}. 
\end{align*}
The sum over $j \geq 0$ converges after possibly shrinking $\epsilon$ to satisfy $\frac{1}{q} - \frac{2}{r} + \epsilon < 0.$
\end{proof}

\begin{proof}[Proof of Proposition \ref{maintheoremforchar}]
For each $\upsilon \in \Upsilon$ let $s_{\upsilon}$ and $d_{\upsilon}$ denote the left and right endpoints, respectively, of the interval $\upsilon.$ Write
\begin{equation}
1_{\upsilon} = \sum_{j} \psi_{s,\upsilon, j} + \psi_{m,\upsilon,j} + \psi_{d,\upsilon,j}
\end{equation}
where $\psi_{s,\upsilon, j}$ is supported on $(s_{\upsilon} + 2^{-(j+1)}, s_{\upsilon} + .99*2^{-(j-1)})$,  $\psi_{d,\upsilon, j}$ is supported on $(d_{\upsilon} - .99*2^{-(j-1)}, d_{\upsilon} - 2^{-(j+1)})$, $\psi_{m,\upsilon, j}$ is supported on $(\frac{d_{\upsilon} + s_{\upsilon}}{2} - .99*2^{-j}, \frac{d_\upsilon + s_\upsilon}{2} + .99*2^{-j})$, and where $\psi_{s,\upsilon, j} = 0$ for $2^{-(j - 1)} > |\upsilon|,$ $\psi_{d,\upsilon, j} = 0$ for $2^{-(j - 1)} > |\upsilon|$ and  $\psi_{m,\upsilon, j} = 0$ when $2^{-(j-1)} > |\upsilon|$ or $2^{-(j-1)} \leq |\upsilon|/2$ (thus, each function is supported on $\upsilon$, the supports of the functions are finitely overlapping, and each function with parameter $j$ is supported on an interval of diameter approximately $2^{-j}$ around an endpoint of $\upsilon$). Furthermore, we require that the $\psi_{s,\upsilon, j}$ are smooth and satisfy  
\begin{equation} \label{psiadapted}
\|\psi^{(M)}_{s,\upsilon, j}\|_{L^\infty} \leq C_M 2^{Mj} 
\end{equation}
for some large $M$ depending on $\epsilon$
and similarly for the functions $\psi_{m, \upsilon, j},$ and $\psi_{d, \upsilon, j}.$\footnote{Following the standard conventions for the addition of extended reals, this decomposition works equally well if $\{s_{\upsilon}, d_{\upsilon}\}$ has one infinite element. If both endpoints are infinite then $|\Upsilon|=1$ and the theorem is already known.}

For Schwartz functions $f$ we have
\[
\Delta^*[\sum_{\upsilon \in \Upsilon} (c_{\upsilon} 1_{\upsilon}\hat{f})\check{\ }] = \sup_k |\Delta_k[\sum_{\upsilon \in \Upsilon} \sum_j (c_{\upsilon} (\psi_{s, \upsilon, j} + \psi_{m, \upsilon, j} + \psi_{d, \upsilon, j}) \hat{f})\check{\ }]|.
\]
By a standard limiting argument it suffices to prove a version of \eqref{mtfcbound} where the supremum in $k$ and the sum in $j$ above only range over finite sets of integers (provided, as usual that the constant is independent of this set). We will further simplify matters by replacing $\psi_{s, \upsilon, j} + \psi_{m, \upsilon, j} + \psi_{d, \upsilon, j}$ by $\psi_{s, \upsilon, j}$; the $\psi_{d, \upsilon, j}$ term is handled through a completely symmetric argument, and obvious minor modifications suffice to bound the $\psi_{m, \upsilon, j}$ term. Henceforth, abbreviate $\psi_{s, \upsilon, j} =: \psi_{\upsilon, j},\ |\Xi| + |\Upsilon| =: N,$ and $\sum_{\upsilon \in \Upsilon} (c_{\upsilon} \psi_{\upsilon, j} \hat{f})\check{\ } =: \Psi_j[f] $. Since $\sum_j \Psi_j$ is
bounded on $L^2$ with norm $\leq C_\epsilon \sup_{\upsilon \in \Upsilon} |c_{\upsilon}|$, an estimate for $\Delta^*\sum_j\Psi_j$ at $q=2$ with norm bounded by 
\[
A := C_{r,\epsilon} (1 + \log N) N^{\frac{1}{2} - \frac{1}{r}} (D_M + \sup_{\xi \in \Xi} \|\sum_{\substack{\omega \in \Omega : |\omega| = 2^k }} \phi_\omega(\xi)\|_{V^r_k}) \sup_{\upsilon \in \Upsilon} |c_{\upsilon}| 
\]
follows immediately from \eqref{LqBourgain}. Thus, by interpolation, it suffices to prove the weak-type 1-1 estimate
\begin{equation} \label{weak11bound}
|\{ x : \sup_k |\Delta_k[\sum_j \Psi_j[f]](x)| > \lambda \}| \leq C_{\epsilon} A N^{\frac{1}{2} + 5 \epsilon} \lambda^{-1} \|f\|_{L^1}.
\end{equation}
For later convenience, assume a renormalization so that $A = 1$ and 
\begin{equation} \label{cupsilonnormal}
\sup_{\upsilon \in \Upsilon}|c_{\upsilon}| = 1.
\end{equation}

We now perform a multiple-frequency Calder\'{o}n-Zygmund decomposition. Specifically, it was shown in \cite{nazarov10czd} that one can write
\[
f = g + \sum_{I \in \calI} b_I
\]
where $\calI$ is a collection of disjoint intervals satisfying
\[
\sum_{I \in \calI} |I| \leq C N^{1/2} \lambda^{-1} \|f\|_{L^1},
\]
where for each $I \in \calI$ and $\xi \in \Xi \cup \{s_{\upsilon} : \upsilon \in \Upsilon\}$, 
\begin{equation}
\|g\|_{L^2}^2 \leq C N^{1/2} \lambda \|f\|_{L^1}
\end{equation}
\begin{equation}
\|f_I\|_{L^1} \leq C N^{-1/2} \lambda |I|
\end{equation}
\begin{equation}
\|b_I - f_I\|_{L^2} \leq C \lambda |I|^{1/2}
\end{equation}
\begin{equation} \label{meanzcondition}
\int b_I(x) e^{-i\xi x}\ dx = 0,
\end{equation}
and where $b_I$ is supported on $3I$, the interval with the same center as $I$ and thrice the diameter. Above we abbreviate $1_If =: f_I$.

Following the Calder\'{o}n-Zygmund method, to establish \eqref{weak11bound} it will suffice to show that for each $I \in \calI$
\[
\|\sup_k |\Delta_k[\sum_j \Psi_j[b_I]]|\|_{L^1((5I)^c)} \leq C_{\epsilon} N^{5\epsilon} \lambda |I|.
\]
By translation and dilation invariance, we may assume $I$ is centered at $0$ with $1/2 < |I| \leq 1.$
Estimating
\begin{equation} \label{mainsplit}
\sup_k |\Delta_k[\sum_j \Psi_j[b_I]]| \leq \sum_j \sup_{k \leq j} |\Delta_k[\Psi_j[b_I]]| + \sum_k |\sum_{j < k} \Psi_j[\Delta_k[b_I]]|
\end{equation} 
we will start by treating the contribution from the first term on the right side above.

We first consider summands with $2^j > N^{-\epsilon}$. 
Then
\begin{multline} \label{L1byL2js}
\|\sup_{k \leq j} |\Delta_k[\Psi_j[b_I]]|\|_{L^1((5I)^c)} \leq \\ C 2^{j(1 + \epsilon)/2}N^{\epsilon} \|\sup_{k \leq j} |\Delta_k[\Psi_j[b_I]]|\|_{L^2} + \|\sup_{k \leq j} |\Delta_k[\Psi_j[b_I]]|\|_{L^1((2^{j(1 + \epsilon)}N^{2\epsilon}5I)^c)}. 
\end{multline}
It follows from (the renormalization of) \eqref{LqBourgain} that
\[
\|\sup_{k \leq j} |\Delta_k[\Psi_j[b_I]]|\|_{L^2} \leq \|\Psi_j[b_I]\|_{L^2}.
\]
Using the modulated mean-zero condition \eqref{meanzcondition} with $\xi \in \{s_{\upsilon} : \upsilon \in \Upsilon\}$ we see 
\begin{align*}
\Psi_j[b_I] &= \sum_{\upsilon \in \Upsilon} c_{\upsilon} \check{\psi}_{\upsilon, j} * b_I - c_{\upsilon} \check{\psi}_{\upsilon, j}\int_{3I}e^{-i s_{\upsilon}y} b_I(y)\ dy  \\
&=: \sum_{\upsilon \in \Upsilon} T_{\upsilon, j}[b_I] \\
&=\sum_{\upsilon \in \Upsilon} T_{\upsilon, j}[f_I] + T_{\upsilon, j}[b_I - f_I].  
\end{align*}
From the decay (by \eqref{psiadapted}) of the derivative of $e^{-i s_{\upsilon}\cdot}\psi_{j,\upsilon}$ and \eqref{cupsilonnormal} one obtains the pointwise estimate (for $h$ supported on $3I$)
\begin{equation} \label{pointwiseest}
|T_{\upsilon, j}[h](x)| \leq C_{\epsilon} 2^{-2j} (1 + \min(1,2^{-j})|x|)^{-2} \|h\|_{L^1(3I)}
\end{equation}
which gives
\begin{equation} \label{L1toL2Tupsilon}
\|T_{\upsilon, j}[h]\|_{L^2} \leq C_{\epsilon} N^{\epsilon/2} 2^{-3j/2} \|h\|_{L^1(3I)}.
\end{equation} 
The orthogonality of $\{T_{\upsilon, j}[h]\}_{\upsilon \in \Upsilon}$ implies
\[
\|\sum_{\upsilon \in \Upsilon}T_{\upsilon, j}[f_I]\|_{L^2} \leq C_{\epsilon} N^{\epsilon/2} 2^{-3j/2} \lambda |I|
\]
which is acceptable when summed over $j$. 

To obtain an $L^2$ bound for the $b_I - f_I$ term, one considers the almost orthogonality of $\{T_{\upsilon, j}\}_{\upsilon \in \Upsilon}$ as operators from $L^2(3I) \rightarrow L^2.$ Recycling \eqref{L1toL2Tupsilon} gives 
\[
\|T_{\upsilon, j}\|_{L^2(3I) \rightarrow L^2} \leq C_{\epsilon} N^{\epsilon/2} 2^{-3j/2}.
\]
Reusing the genuine orthogonality, one obtains
\[
\|T^*_{\upsilon', j}T_{\upsilon, j}\|_{L^2(3I) \rightarrow L^2(3I)} = 0
\]
when $\upsilon' \neq \upsilon$.
Integrating by parts once (see \cite{nazarov10czd} for details) and arguing as in \eqref{pointwiseest} gives
\[
\|T_{\upsilon', j}T^*_{\upsilon, j}\|_{L^2 \rightarrow L^2} \leq C_{\epsilon} \frac{1}{|s_{\upsilon} - s_{\upsilon'}|} N^{\epsilon} 2^{-3j} 
\]
for $\upsilon' \neq \upsilon.$ Whenever $\psi_{\upsilon, j}$ and $\psi_{\upsilon', j}$ are nonzero we have $|s_{\upsilon} - s_{\upsilon'}| \geq 2^{-j}.$ Thus for each $\upsilon$
\[
\sum_{\upsilon' \in \Upsilon} (\|T_{\upsilon', j}T^*_{\upsilon, j}\|_{L^2 \rightarrow L^2})^1 \leq C_{\epsilon} (1 + \log(N)) N^{2\epsilon} 2^{-2j},
\]
\[
\sum_{\upsilon' \in \Upsilon} (\|T_{\upsilon', j}^*T_{\upsilon, j}\|_{L^2(3I) \rightarrow L^2(3I)})^0 \leq 1,
\]
and hence one can apply a weighted version of the Cotlar-Stein lemma (see \cite{comech07csa}, or use an alternative argument as in \cite{nazarov10czd}) to conclude that
\[
\|\sum_{\upsilon} T_{\upsilon, j}\|_{L^2(3I) \rightarrow L^2} \leq C_{\epsilon} (1 + \log(N))^{1/2} N^{\epsilon} 2^{-j}. 
\]
Summing over $j$, this gives an acceptable contribution from $b_I - f_I.$

Proceeding to the second term on the right of \eqref{L1byL2js} we (again using \eqref{cupsilonnormal}) estimate
\[
\|\sup_{k \leq j} |\Delta_k[\Psi_j[b_I]]|\|_{L^1((2^{j(1 + \epsilon)}N^{2\epsilon}5I)^c)} \leq \sum_{\upsilon \in \Upsilon} \|\sup_{k \leq j} |\Delta_k[\check{\psi}_{\upsilon, j}*b_I]|\|_{L^1((2^{j(1 + \epsilon)}N^{2\epsilon}5I)^c)}.
\]
For each $\upsilon$ and $k \leq j$ the number of dyadic intervals of length $2^{-k}$ which intersect the support of $\psi_{\upsilon, j}$ is at most 3. Thus
\[
\sup_{k \leq j} |\Delta_k[\check{\psi}_{\upsilon, j}*b_I]|\ \leq 3 \sup_{k \leq j} \sup_{|\omega| = 2^{-k}} |\check{\phi}_{\omega} * \check{\psi}_{\upsilon,j} * b_I|. 
\] 
Using \eqref{necessarysmoothness} and \eqref{psiadapted} (and the normalization which ensures $D_M \leq 1$) one obtains the estimate
\[
|\check{\phi}_{\omega} * \check{\psi}_{\upsilon,j}(x)| \leq C_{\epsilon} 2^{-j}(1 + |2^{-j}x|)^{-M}
\]
uniformly in $k \leq j$ and $|\omega| = 2^{-k}$. This gives 
\[
\sup_{k \leq j} \sup_{|\omega| = 2^{-k}} |\check{\phi}_{\omega} * \check{\psi}_{\upsilon,j} * b_I| \leq C_{\epsilon} 2^{-j}(1 + |2^{-j}\cdot|)^{-M}* |b_I| .
\]
Since $b_I$ is supported on $3I$ we thus conclude
\[
 \sum_{\upsilon \in \Upsilon} \|\sup_{k \leq j} |\Delta_k[\check{\psi}_{\upsilon, j}*b_I]|\|_{L^1((2^{j(1 + \epsilon)}N^{2\epsilon}5I)^c)} \leq N C_{\epsilon} N^{-2\epsilon(M-1)}2^{-j\epsilon(M-1)} \|b_I\|_{L^1}.
\]
Taking $M \geq 1 + 1/(2\epsilon - \epsilon^2))$, the sum over $j$ of the right side above is $\leq C_{\epsilon} \lambda |I|$ as desired.

The case $2^j \leq N^{-\epsilon}$ is covered by a reiteration of the argument in the preceding paragraph.

To bound
\[
\|\sum_k |\sum_{j < k} \Psi_j[\Delta_k[b_I]]|\|_{L^1((5I)^c)}
\]
one argues as above, except with roles of $\Delta_k$ and $\Psi_j$ interchanged. Specifically, one now uses the modulated mean-zero condition with $\xi \in \Xi$ to obtain the $L^2$ estimate. For the remaining terms, we rely on the fact that for each $\omega$ with $|\omega| = 2^{-k}$ there are at most 5 pairs $(\upsilon, j)$ with $j < k$, and the support of $\psi_{\upsilon, j}$ intersecting the interval $\omega.$ Thus, for each $k$
\[
|\sum_{j < k} \Psi_j[\Delta_k[b_I]]| \leq C_{\epsilon} 5N 2^{-k}(1 + |2^{-k} \cdot|)^{-M}*|b_I|.
\] 
\end{proof}

\section{Variation-norm estimates} \label{vnsection}

We will now prove variation-norm analogues of Theorems \ref{maintheorem} and \ref{onefrequencytheorem}.

\begin{theorem} \label{varonefrequencytheorem}
Let $\phi$ be a Schwartz function such that $\hat{\phi}$ is compactly supported, let $r > 2,$ and let $V_{\phi}$ be the associated $r$-variation norm operator
\[
V_{\phi}[f](x) = \|\int f(x - y) 2^{-k}\phi(2^{-k}y)\ dy\|_{V^r_k}.
\]
Then 
\begin{equation} 
\|V_{\phi}[H[f]]\|_{L^{1,\infty}} \leq C_{\phi} \|f\|_{L^1}
\end{equation}
where $\|\cdot\|_{L^{1,\infty}}$ denotes the weak $L^1$ Lorentz norm. 
\end{theorem}

\begin{proof}
Since $r > 2$, we have $V_{\phi}$ bounded on $L^2$ (see, for example, \cite{jones08svj}). Following the proof and notation from Theorem \ref{onefrequencytheorem} it thus remains to estimate, for $x \in (3I)^c$ and each $j$
\[
\|\phi_k*\psi_j*b(x)\|_{V^r_{k < j}}.
\]
Fix a sequence $k_0 < \ldots < k_L < j$ and consider 
\begin{equation} \label{sfvarpoint}
\sum_{l = 1}^L |\phi_{k_l}*\psi_j*b(x) - \phi_{k_{l-1}} * \psi_{j}*b(x)|.
\end{equation}
For each $l$ we have $\hat{\phi}_{k_l}(0) - \hat{\phi}_{k_{l-1}}(0) = 0$ and 
\[
\|(\hat{\phi}_{k_l} - \hat{\phi}_{k_{l-1}})'\|_{L^{\infty}} \leq C 2^{k_{l}}
\]
which implies that $|\hat{\phi}_{k_l} - \hat{\phi}_{k_{l-1}}| \leq C 2^{k_l - j}$ on the support of $\hat{\psi}_j$. This gives
\[
\| ((\hat{\phi}_{k_l} - \hat{\phi}_{k_{l-1}}) \hat{\psi}_j)^{(m)}\|_{L^{\infty}} \leq C 2^{k_l - j} 2^{m j}
\]
for $m = 0,2$ and so 
\begin{equation} \label{sfvarphi}
|(\phi_{k_l} - \phi_{k_{l-1}}) * \psi_j(x)| \leq C 2^{k_{l} - j} 2^{-j}(1 + |2^{-j}x|)^{-2} 
\end{equation}
and
\begin{equation} \label{sfvarphiprime}
|\frac{d}{dx}((\phi_{k_l} - \phi_{k_{l-1}})* \psi_j)(x)| \leq C 2^{k_{l} - j} 2^{-2j}(1 + |2^{-j}x|)^{-2}.
\end{equation}
From \eqref{sfvarphi} one sees that for each $x$ and $j$ \eqref{sfvarpoint} is
\[
\leq C 2^{-j}(1 + |2^{-j}\cdot|)^{-2} * |b|(x)
\]
and that for $2^j > |I|$ \eqref{sfvarpoint} is
\[
\leq (|I|/2^{j}) 2^{-j}(1 + |2^{-j}\cdot|)^{-2} * |b|(x)
\]
as desired.

\end{proof}

\begin{proof}[Proof of Theorem \ref{varmaintheorem}]
As in Theorem \ref{maintheorem}, the proof follows immediately from a suitable version of Proposition \ref{maintheoremforchar}; we retain the notation therein. From a result in \cite{nazarov10czd} we see that the estimate at $q=2$ holds with norm 
\[
A =  C_{r,s}  
(1 + \log N) N^{(\frac{1}{2} - \frac{1}{r})\frac{s}{s-2}}\\
 (D_M + \sup_{\xi \in \Xi} \|\sum_{\substack{\omega \in \Omega : |\omega| = 2^k }} \phi_\omega(\xi)\|_{V^r_k}) \sup_{\upsilon \in \Upsilon} \|\psi_{\upsilon}\|_{V^r}
\]  
which, again, we renormalize to $1$. The proof of Proposition \ref{maintheoremforchar} then carries through except for the treatment of the second term on the right side of \eqref{L1byL2js} and the terms $2^j \leq N^{-\epsilon}$ from the first term on the right side of \eqref{mainsplit}. In both situations we must consider, for fixed $\upsilon, j$ and $x$
\[
\|\Delta_k[\check{\psi}_{\upsilon, j} * b_I](x)\|_{V^s_{k \leq j}}.
\]
Let $\tilde{\Omega}$ be the set of minimal dyadic intervals in 
\[
\{\omega \in \Omega : \omega \cap \Xi \neq \emptyset, \omega \cap (s_{\upsilon}, s_{\upsilon} + 2^{-(j-1)}) \neq \emptyset, \text{\ and\ } |\omega| \geq 2^{-j}\}.
\] 
Then $|\tilde{\Omega}| \leq 3.$ Let $\tilde{\Xi} \subset \Xi$ be chosen so that $|\tilde{\Xi}| = |\tilde{\Omega}|$ and so that for each $\omega \in \tilde{\Omega}$, $\omega \cap \tilde{\Xi} \neq \emptyset.$ Finally, for each $\xi \in \tilde{\Xi}$ let $k_{\xi}$ be chosen so that the interval in $\tilde{\Omega}$ containing $\xi$ has length $2^{-k_{\xi}}.$ Then
\[
\|\Delta_k[\check{\psi}_{\upsilon, j} * b_I](x)\|_{V^s_{k \leq j}} \leq C \sum_{\xi \in \tilde{\Xi}} \|\check{\phi}_{\omega_{\xi, k}} * \check{\psi}_{\upsilon, j} * b_I\|_{V^s_{k \leq k_{\xi}}}
\]
where $\omega_{\xi,k}$ is the dyadic interval of length $2^{-k}$ containing $\xi.$
For each $\xi \in \tilde{\Xi}$, the support of $\psi_{\upsilon, j}$ is distance $\leq C 2^{-k_\xi}$ away from $\xi.$ Thus, 
given any $k_{l-1} < k_l \leq k_{\xi}$ we have
\[
|\phi_{\omega_{\xi, k_{l - 1}}}(\eta) - \phi_{\omega_{\xi, k_l}}(\eta)| \leq |\phi_{\omega_{\xi, k_{l - 1}}}(\xi) - \phi_{\omega_{\xi, k_l}}(\xi)| + C 2^{k_{l} - k_{\xi}}
\]
for $\eta$ in the support of $\psi_{\upsilon, j}.$ It then follows that 
\[
|(\check{\phi}_{\omega_{\xi, k_{l - 1}}} - \check{\phi}_{\omega_{\xi, k_l}})*\check{\psi}_{\upsilon, j}| \leq C (|\phi_{\omega_{\xi, k_{l - 1}}}(\xi) - \phi_{\omega_{\xi, k_l}}(\xi)| + 2^{k_{l} - k_{\xi}}) 2^{-j}(1 + |2^{-j}\cdot|)^{-M}
\]
and so 
\[
\|\check{\phi}_{\omega_{\xi, k}} * \check{\psi}_{\upsilon, j} * b_I(x)\|_{V^s_{k \leq k_{\xi}}} \leq C (\|\phi_{\omega_{\xi, k}}(\xi)\|_{V^s_{k \leq k_{\xi}}} + \sum_{k \leq k_{\xi}}2^{k - k_{\xi}}) 2^{-j}(1 + |2^{-j} \cdot|)^{-M} * |b_I|(x).
\]
This gives the desired bound, since by our normalization 
\[
\|\phi_{\omega_{\xi, k}}(\xi)\|_{V^s_{k \leq k_{\xi}}} \leq \sup_{\xi \in \Xi} \|\sum_{\substack{\omega \in \Omega : |\omega| = 2^k }} \phi_\omega(\xi)\|_{V^r_k} \leq 1.
\]
\end{proof}


\bibliographystyle{hamsplain}
\bibliography{roberlin}

\end{document}